\documentclass[a4paper,11pt]{article}

\usepackage{a4wide}
\usepackage{amsfonts}
\usepackage{graphics}
\usepackage{amsmath, amsthm}  
\usepackage{amssymb,bbm} 
\usepackage{enumerate}
\usepackage{amsthm}
\usepackage{hyperref}
\usepackage{cite}
\usepackage[all,cmtip]{xy}

\hypersetup{
	bookmarks=true,         
	unicode=false,          
	pdftoolbar=true,        
	pdfmenubar=true,        
	pdffitwindow=true,     
	pdfstartview={FitH},    
	pdftitle={Solvable groups and identities of automorphisms},    
	pdfauthor={Wolfgang Alexander Moens},     
	pdfsubject={Survey},   
	pdfkeywords={Finite groups, Solvable groups, Automorphisms, Polynomial identities}, 
	pdfnewwindow=true,      
	colorlinks=true,       
	linkcolor=gray,          
	citecolor=red,        
	filecolor=magenta,      
	urlcolor=gray           
}

\usepackage{tikz}
\usepackage{epstopdf}

\newcommand{\id}{\mathbbm{1}}\newcommand{\N}{\mathbb{N}}
\newcommand{\Q}{\mathbb{Q}}\newcommand{\Z}{\mathbb{Z}}\newcommand{\F}{\mathbb{F}}

\newcommand{\HI}{\operatorname{inv}}
\newcommand{\HL}{\operatorname{len}}
\newcommand{\irred}{\operatorname{irr}}
\newcommand{\constA}{{\rho_1(f(x))}}
\newcommand{\constB}{{\rho_2(f(x))}}
\newcommand{\constC}{{\rho_3(f(x))}}
\newcommand{\auxpol}{f_\ast}
\newcommand{\auxpolg}{\overline{f}}
\newcommand{\auxpoldn}{\overline{f}}

\newcommand{\map}[3]{ #1 : #2 \longrightarrow #3 }

\newcommand{\mapl}[5]{ #1 : #2 \longrightarrow #3 : #4 \longmapsto #5 }

\newcommand{\discr}[1]{ \operatorname{Discr}_\ast(#1) }
\newcommand{\prodant}[1]{ \operatorname{Prod}(#1) }

\newtheorem{theorem}{Theorem}[section]

\newtheorem{proposition}[theorem]{Proposition}

\newtheorem*{corollary*}{Corollary}

\newtheorem{lemma}[theorem]{Lemma}
\newtheorem{claim}[theorem]{Claim}
\newtheorem*{theorem*}{Theorem}
\newtheorem*{maintheorem*}{Main Theorem}
\newtheorem*{definition*}{Definition}
\newtheorem*{remark*}{Remark}

\newtheorem*{proposition*}{Proposition}
\newtheorem*{example*}{Example}
\newtheorem*{lemma*}{Lemma}
\newtheorem*{problem*}{Problem}

\newtheorem{openproblem*}{Open Problem}

\newtheorem*{observation*}{Observation}

\newtheorem{conjecture}[theorem]{Conjecture}
\theoremstyle{definition} \newtheorem{remark}[theorem]{Remark}
\theoremstyle{definition} \newtheorem{example}[theorem]{Example}
\theoremstyle{definition}\newtheorem{definition}[theorem]{Definition}
\theoremstyle{definition}

\title{The Fitting height of finite groups with a \\fixed-point-free automorphism satisfying an identity
	}
\author{Wolfgang Alexander Moens\footnote{This research was supported by the Austrian Science Fund (FWF) $P30842 - N35$.}}
\date{}

\begin{document}
	
\maketitle

\abstract{
	Motivated by classic theorems of Thompson and Berger on the Fitting height of finite groups with a fixed-point-free automorphism of coprime order, we conjecture that, for every non-zero polynomial $f(x) = a_0 + a_1 x + \cdots + a_d x^d \in \Z[x] $, there is an integer $k > 0$ with the following property. Let $G$ be a finite (solvable) group with a fixed-point-free automorphism $\alpha$ satisfying $\gcd(|G|,k)= 1$ and $$\{ g^{a_0} \cdot \alpha(g)^{a_1} \cdot \alpha^2(g)^{a_2} \cdots \alpha^d(g)^{a_d} | g \in G \} = \{1\}.$$ Then the Fitting height of $G$ is at most the number of irreducible factors of $f(x)$. We confirm the conjecture for a large family of polynomials with explicit constants $k$.
}

\section{Introduction}

\paragraph{Classic results.} Let $G$ be a finite group admitting a fixed-point-free automorphism $\alpha$. Such a group $G$ is necessarily solvable, according to the classification of the finite simple groups \cite{Rowley}. It is not known whether the derived length of $G$ can be bounded from above by a function that depends only on the order $n := |\alpha|$ of the automorphism. But Dade \cite{Dade} was able to find such a function bounding the Fitting height of $G$ and, in doing so, solved a conjecture of Thompson \cite{ThompsonFitting}. Jabara \cite{JabaraFitting} has recently lowered the bound on the Fitting height to $7 \cdot \Omega(n)^2$, where $\Omega(n)$ counts the number of prime divisors of $n$ with multiplicity. \\

Under the additional assumption that $\gcd(|G|,n) = 1$, even stronger results have been obtained. A theorem of Berger \cite{Berger}, refining an earlier result of Thompson \cite{ThompsonFitting}, states that the Fitting height of $G$ is then at most ${\Omega(n)}$. Examples of Gross \cite{GrossFittingEx} show moreover that this bound is optimal. In the special case that the automorphism has prime order $n$, Berger's theorem states that the group $G$ is nilpotent. This result, also originally due to Thompson \cite{ThompsonFixPointFreeAutomorphisms}, gave a positive answer to the so-called Frobenius conjecture. In this case, a theorem by Higman \cite{HigmanGroupsAndRings} shows that even the nilpotency class of $G$ can be bounded from above by a function that depends only on $n$. Kreknin and Kostrikin \cite{KrekninKostrikinRegularAutomorphisms} later found the explicit upper bound $(n-1)^{2^{(n-1)}}$. \\ 

In order to discuss various analogues of these results, we recall a definition from \cite{MoensIdentitiesGroups}. Let $H$ be a group and let $\gamma$ be an endomorphism of $H$. The polynomial $f(x) = a_0 + a_1 \cdot x + a_2 \cdot x^2 + \cdots + a_n \cdot x^n \in \Z[x]$ is an \emph{ordered identity} of $\gamma$ if the map $$\mapl{f(\gamma)}{H}{H}{h}{h^{a_0} \cdot \gamma(h)^{a_1} \cdot \gamma^2(h)^{a_3} \cdots \gamma^n(h)^{a_n}}$$ vanishes identically. 

\begin{example} \label{Ex1}
	Let $G$ be a finite group with automorphism $\alpha$. If the polynomial $-1 + x^n$ is an ordered identity of $\alpha$, then the order of $\alpha$ divides $n$. And vice-versa.
\end{example}

So we conclude that the Fitting height of a finite group admitting a fixed-point-free automorphism with ordered identity $-1 + x^n$ is bounded from above by $7 \cdot \Omega(n)^2$ and that it is bounded by $\Omega(n)$ if $\gcd(|G|,n) = 1$. 

\begin{example} \label{Ex2}
	Let $G$ be a finite group with endomorphism $\gamma$. If the constant polynomial $n$ is an ordered identity of $\gamma$, then the exponent of $G$ divides $n$. And vice-versa.
\end{example}

Finite groups of bounded exponent have been studied extensively in the context of the restricted Burnside problem, most notably by Hall---Higman \cite{HallHigmanReduction}, Kostrikin \cite{KostrikinBurnsideProblem}, and Zel'manov \cite{EfimBurnsideProblemEven,EfimBurnsideProblemOdd}. It is known, in particular, that the Fitting height of a finite solvable group with ordered identity $n \in \Z \setminus \{0\}$ is bounded from above by a function that depends only on the prime factorization $\prod_i p_i^{k_i}$ of $n$. In Shalev's note \cite{ShalevCentralizersResiduallyFiniteTorsion}, one can find the upper bound $\prod_i (2 k_i + 1) $, which readily implies the upper bound $ 3^{\Omega(n)}$.

\begin{example} \label{Ex3}
	Let $G$ be a finite group with endomorphism $\gamma$. If the linear polynomial $-n + t$ is an ordered identity of $\gamma$, then $G$ is an $n$-abelian group. And vice-versa.
\end{example}

These $n$-abelian groups were studied by Baer \cite{Baer} and classified by Alperin \cite{AlperinPowerAutomorphism}. The classification implies that, for $n \not \in \{0,1\}$, the Fitting height of a finite, $n$-abelian group $G$ is at most $\max \{ 3^{\Omega(n)} , 3^{\Omega(n-1)}\}$. If we further assume that $\gcd(|G|,n(n-1)) = 1$, then $G$ is known to be abelian, so that the Fitting height of $G$ is at most $1$. 

\begin{example} \label{Ex4}
	Let $G$ be a finite group with automorphism $\alpha$. The automorphism $\alpha$ is said to be $n$-split if and only if $1+x+x^2+\cdots +x^{n-1}$ is an ordered identity of $\alpha$.
\end{example}

It is not difficult to verify that the automorphism in Example \ref{Ex1} is $n$-split if it is fixed-point-free. And, conversely, if the automorphism in Example \ref{Ex4} is fixed-point-free, then its order divides $n$. But finite groups $G$ with an $n$-split automorphism that is \emph{not} necessarily fixed-point-free have also been studied extensively in the literature. Ersoy \cite{ErsoySplittingAutomorphisms} has shown that such a group $G$ is solvable, provided that $n$ is odd. If $n$ is a prime, then $G$ is even nilpotent, according to a theorem of Hughes---Thompson \cite{HughesThompson} and Kegel \cite{KegelSplit}. We refer to the work of Khukhro \cite{KhukhroSplitPrimeFiniteGroup,KhukhroCompact} and Zel'manov \cite{ZelmanovPlatonovConjecture} for more results involving $n$-split automorphisms. We also highlight Zel'manov's recent generalization of these results to torsion groups with identities \cite{ZelmanovAlgebrasTorsionGroupsIdentity}. \\

More theorems in this general spirit can be found in the literature. We refer, in particular, to Turull's classic results \cite{TurullFittingHeight} on the Fitting height of finite groups with a fixed-point-free group of coprime operators, and to the recent results in \cite{ErcanGulinGuloglu,ErcanGulogluJabara}.

\paragraph{A change in perspective.} For any given polynomial $f(x) \in \Z[x] \setminus \{0\}$, we now consider all the finite groups admitting a fixed-point-free automorphism with this ordered identity $f(x)$. We claim that it is possible to find a \emph{uniform} upper bound on the Fitting height of such groups, and that we can find a particularly good bound if we further exclude finitely-many primes from the torsion in those groups. 

\begin{conjecture} \label{Conjecture_Main}
	For every $f(x) \in \Z[x] \setminus \{0\}$, there exist integers $\operatorname{k}(f(x)) > 0$ and $\operatorname{m}(f(x)) $ with the following property. \\
	
	Let $G$ be a finite (solvable) group admitting a fixed-point-free automorphism $\alpha$ and let $f(x)$ be an ordered identity of $\alpha$.
	\begin{enumerate}[\rm (a)]
		\item Then the Fitting height of $G$ is at most $\operatorname{m}(f(x))$.
		\item If $\gcd(|G|,\operatorname{k}(f(x))) = 1$, then the Fitting height of $G$ is at most the number of irreducible factors of $f(x)$.
	\end{enumerate}
\end{conjecture}

The classic theorems outlined above settle the conjecture for all polynomials of the form $-1 + x^n$, for all constant polynomials $n$, for all polynomials of the form $-n + x$, and for all polynomials of the form $(x^n-1)/(x-1)$. Moreover, in \cite{MoensIdentitiesGroups}, we settled claim (b) of the conjecture for all {irreducible} polynomials. The results in \cite{MoensIdentitiesGroups} also extend the already-mentioned results of Alperin \cite{AlperinPowerAutomorphism}, Thompson \cite{ThompsonFixPointFreeAutomorphisms}, Higman \cite{HigmanGroupsAndRings}, Kreknin---Kostrikin \cite{KrekninKostrikinRegularAutomorphisms}, Hughes---Thompson \cite{HughesThompson}, Kegel \cite{KegelSplit}, and Khukhro \cite{KhukhroSplitPrimeFiniteGroup} in various ways.

\paragraph{Main result.} In this paper, we will confirm claim (b) of the conjecture for the family of all \emph{Higman-solvable} polynomials. These polynomials are formally introduced in Definition \ref{Definition_HigmanSolvable} and we postpone the technical details to Section \ref{Section_Invariants}. Informally, a polynomial $f(x) \in \Z[x]$ is said to be Higman-solvable if $f(0) \cdot f(1) \neq 0$ and if $f(x)$ admits a decomposition of the form $f(x) = g_1(x^{n_1 \cdots n_c}) \cdots g_{c-1}(x^{n_{c-1} \cdot n_c}) \cdot g_c(x^{n_c})$ with some additional conditions the polynomials $g_i(x)$. The following proposition shows that these conditions are satisfied \emph{with high probability}, in the sense of probabilistic Galois theory. \\

We recall that the \emph{height} of a polynomial is the maximal modulus of its coefficients. 

\begin{proposition} \label{Proposition_Generic}
	Fix $c,d_1,n_1, \ldots, d_c,n_c \in \Z_{\geq 1}$. For each $h \in \Z_{\geq 1}$, we consider the $c$-tuples $(g_1(x),\ldots,g_c(x))$ of monic polynomials in $\Z[x]$ with prescribed degrees $\deg(g_1(x)) = d_1 , \ldots , \deg(g_c(x)) = d_c$ and heights at most $h$. Let $P(h)$ be the probability that the product \begin{equation*}
		f(x) = g_1(x^{n_1 \cdots n_c}) \cdots g_{c-1}(x^{n_{c-1} \cdot n_c}) \cdot g_c(x^{n_c}) \label{Eq_decomp}
	\end{equation*} is Higman-solvable. Then $\lim_{h \rightarrow + \infty} P(h) = 1$.
\end{proposition}

We now turn to our main result. For a given Higman-solvable polynomial $f(x) \in \Z[x]$, we let $\HI(f(x))$ be the non-zero, integer invariant of Definition \ref{Definition_MainInvariant}. This $\HI(f(x))$ will play the role of the invariant $\operatorname{k}(f(x))$ of Conjecture \ref{Conjecture_Main}. We also let $\irred(f(x))$ be the number of irreducible factors of $f(x)$ of positive degree.

\begin{theorem} \label{Theorem_Main}
	Let $G$ be a finite (solvable) group admitting a fixed-point-free automorphism $\alpha$ and let $f(x)$ be any ordered identity of $\alpha$ that is Higman-solvable. If $\gcd(|G|,\HI(f(x))) = 1$, then the Fitting height of $G$ is at most $\irred(f(x)).$
\end{theorem}

Let us briefly discuss the assumptions. Every monic, Higman-solvable polynomial of positive degree is an ordered identity of a fixed-point-free automorphism of a finite, non-trivial group (Remark \ref{Remark_Auto}). Moreover, every fixed-point-free automorphism of a finite, non-trivial group has a monic, ordered identity that is Higman-solvable (Remark \ref{Remark_Ident}). \\ 

The strategy to prove Theorem \ref{Theorem_Main} is straight-forward. First, we replace the ordered identities with the weaker notion of \emph{abelian identities}. We then show that a natural power of the automorphism induced on the first lower Fitting subgroup $\underline{\operatorname{{F}}_{1}}(G)$ of $G$ satisfies an abelian identity with strictly fewer irreducible factors of positive degree. Our assumption on the prime divisors of $|G|$ will guarantee that the induced automorphism is still fixed-point-free. After at most $\irred(f(x))$ iterations of this argument, we arrive at an abelian identity that is constant, say $e$. Since $|G|$ is assumed to be coprime to this constant $e$, we may conclude that the $\irred(f(x))$'th lower Fitting subgroup of $G$ is the trivial group. We note that this strategy works if and only if the polynomial $f(x)$ is Higman-solvable. \\

In order to make the strategy precise, we will introduce a number of rather technical invariants of $f(x)$ in $\Z$ and $\Z[x]$. We list some of these in Table \ref{Table_Invariants} for the convenience of the reader. The last column illustrates the special case $f(x) := \Phi_p(x)$, where $p$ is a prime and where $\Phi_p(x) = (x^p - 1)/(x-1) = 1 + x + x^2 + \cdots + x^{p-1}$ is the cyclotomic polynomial that vanishes on the primitive $p$'th roots of unity. As we had observed in the examples above, this polynomial corresponds with the Frobenius conjecture. 

\begin{table}[h] \caption{Invariants of a Higman-solvable polynomial $f(x) \in \Z[x]$.} \label{Table_Invariants}
	\begin{center}
		\begin{tabular}{|c||c|c|c|}
			\hline
			\textbf{Auxiliary Invariant} & \textbf{Range} & \textbf{Definition} & $f(x) := \Phi_p(x)$ \\
			\hline \hline 
			$\auxpolg(x) ; \auxpol(x)$ & $\Z[x] \setminus \{0\}$ & \ref{Definition_dhg||}  & $\Phi_p(x) ; 1$ \\
			\hline
			$\Delta(f(x))$ & $\Z[x] \setminus \{0\}$ & \ref{Definition_Delta}  & $1$ \\ \hline 
			$\auxpoldn_n(x)$ & $\Z[x] \setminus \{0\}$ & \ref{Definition_gni_dn}  & $1$\\
			\hline
			$\auxpol^2(x)$ & $\Z[x] \setminus \{0\}$ & \ref{Definition_d2} & $1$\\  
			\hline 
			$\|f(x)\|$ & $\Z_{\geq 0}$ & \ref{Definition_dhg||}  & $1$\\
			\hline
			$\constA$ & $\Z \setminus \{0\}$ & \ref{Definition_ConstA} & $1$\\
			\hline
			$\constB$ & $\Z \setminus \{0\}$ & \ref{Definition_ConstB} & $1$\\
			\hline
			$\constC$ & $\Z \setminus \{0\}$ & \ref{Definition_ConstC}  & $1$\\
			\hline 
			$\discr{f(x)} ; \prodant{f(x)}$ & $\Z \setminus \{0\}$ & $2.1.4$ in \cite{MoensIdentitiesGroups}  & $p^{p-2} ; p^{p-1}$ \\
			\hline \hline 
			\textbf{Main Invariant} & \textbf{Range} & \textbf{Definition} & $f(x) := \Phi_p(x)$ \\
			\hline \hline 
			$\HL(f(x))$ & $\Z_{\geq 0}$ &  \ref{Definition_HigmanSolvableLength} & $1$ \\
			\hline
			$\irred(f(x))$ & $\Z_{\geq 0}$ & \ref{Definition_Irred} & $1$ \\
			\hline
			$\HI(f(x))$ & $\Z \setminus \{0\}$ & \ref{Definition_MainInvariant} & $p^{2(p-1)}$\\
			\hline
		\end{tabular}
	\end{center}
\end{table}


\paragraph{Overview.} In Section \ref{Section_Prelim}, we collect preliminary results from the literature. In Section \ref{Section_Invariants}, we define our auxiliary invariants and Higman-solvable polynomials. We also prove a more precise version of Proposition \ref{Proposition_Generic}. In Section \ref{Section_Aux} we work out the technical aspects of our strategy to prove the main theorem. In Section \ref{Section_Main}, we prove Theorem \ref{Theorem_Main} with weaker assumptions and with stronger conclusions. In Section \ref{Secition_RemarksExamples}, we illustrate our techniques with three well-chosen examples. This paper extends techniques in \cite{HigmanGroupsAndRings} and \cite{MoensIdentitiesGroups}, but we note that it can be read independently of those works.

\section{Preliminaries} \label{Section_Prelim}

\begin{theorem}[Rowley \cite{Rowley}] \label{Theorem_Rowley}
	If a finite group $G$ admits a fixed-point-free automorphism, then $G$ is solvable. 
\end{theorem}

\begin{lemma}[{\cite[Proposition~3.1.3]{MoensIdentitiesGroups}}] \label{Lemma_IdealOfIdentities}
	Let $G$ be a group with an automorphism $\alpha$. Then the identities of $\alpha$ form an ideal of $\Z[x]$.
\end{lemma}

\begin{lemma}[{\cite[Lemma~4.1.1]{MoensIdentitiesGroups}}] \label{Lemma_UniformHallQuotient}
	Let $G$ be a finite (solvable) group and let $\map{\alpha}{G}{G}$ be a fixed-point-free automorphism. $(i.)$ The map $\mapl{-1+\alpha}{G}{G}{g}{g^{-1} \cdot \alpha(g)}$ is bijective. $(ii.)$ If $H$ is a Hall-subgroup of $G$, then some conjugate $K$ of $H$ satisfies $\alpha(K) = K$. $(iii.)$ If $N$ is a normal subgroup of $G$ with $\alpha(N) = N$, then the induced automorphism $\map{\overline{\alpha}}{G/N}{G/N}$ is also fixed-point-free.
\end{lemma}

\begin{definition}[{\cite[Definition~2.1]{MoensAF}}]
	Let $X$ be a subset of the group $G$. We say that $X$ is \emph{arithmetically-free} if $X$ contains no subset of the form  $\{ b, a , a \cdot b , a \cdot b^2 , a \cdot b^{3} , \cdots \}$. We say that $X$ is \emph{product-free} if it does not contain a subset of the form $\{ b , a , a \cdot b\}$.
\end{definition}

\begin{theorem}[Moens {\cite[Theorem~3.14]{MoensAF}}] \label{Theorem_AF}
	Let $A$ be a group and let $X$ be a finite, arithmetically-free subset of $A$. Then there exists a minimal $\operatorname{H}(X,A) \in \Z_{\geq 0}$ with the following property. If $L$ is a Lie ring that is graded by $A$ and supported by $X$, then $L$ is nilpotent of class at most $\operatorname{H}(X,A)$.
\end{theorem}

\begin{example}
	The finite subset $X$ of $A$ is product-free if and only if $\operatorname{H}(X,A) \leq 1$.
\end{example}

\begin{theorem}[Moens {\cite[Theorem~4.3.1]{MoensIdentitiesGroups}}] \label{Theorem_BoundField}
	Let $X$ be a finite, arithmetically-free subset of the multiplicative group $(\F^\times,\cdot)$ of a field $\F$. Then $\operatorname{H}(X,\F^\times) \leq |X|^{2^{|X|}}$.
\end{theorem}

\begin{theorem}[Moens {\cite[Theorem~4.2.2]{MoensIdentitiesGroups}}] \label{Theorem_Embedding}
	Let $L$ be a finite Lie ring with an automorphism $\beta$ with an identity $f(x) \in \Z[x]$ with root set $X$. Suppose that $\gcd(|L|,\discr{f(x)} \cdot \prodant{f(x)}) = 1$. Then there exists an embedding $\map{\varepsilon}{\discr{f(x)} \cdot L}{K}$ of the Lie ring $\discr{f(x)} \cdot L$ into a Lie ring $K$ that is graded by $(\overline{\Q}^\times,\cdot)$ and supported by $X$.
\end{theorem}

\begin{definition}
	For $d,h \in \Z_{\geq 1}$, we define $E_d(h)$ to be the number of monic polynomials in $\Z[x]$ of degree $d$ and height at most $h$ for which the Galois group is \emph{not} the full symmetric group on $d$ letters.
\end{definition}

\begin{theorem}[van der Waerden {\cite{VanDerWaerdenGalois}}] \label{Theorem_vanderWaerden}
	For each $d \in \Z_{\geq 1}$, we have $E_d(h) = o(h^d)$.
\end{theorem}

So the number of monic polynomials of degree $d$ and height at most $h$ that are reducible or decomposable is $o(h^d)$. We refer to Bhargava's recent solution \cite{Bhargava} of the van der Waerden conjecture for an optimal estimate on the growth rate of $E_d(h)$.

\section{Higman-solvable polynomials and their properties} \label{Section_Invariants}

\subsection{Higman-solvable polynomials $f(x)$ and $\HL(f(x))$}

\begin{definition}
	Let $a(x) , b(x) \in \Z[x]$. We say that $a(x)$ is \emph{powerful} if it is in the set $\Z[x^2] \cup \Z[x^3] \cup \Z[x^4] \cdots$. And $a(x)$ is a \emph{closed divisor} of $b(x)$ if $a(x)$ is a divisor of $b(x)$ and 
	\begin{enumerate}[\rm (a)]
		\item every powerful divisor of $b(x)$ divides $a(x)$,
		\item if $A$ is the root set of $a(x)$ and $B$ is the root set of $b(x)$, then $A \cdot A \cap B \subseteq A$, and 
		\item $\gcd(a(x) , b(x) / a(x)) = 1$. 
	\end{enumerate}
\end{definition}

It is clear that the gcd of all closed divisors of $a(x)$ is again a closed divisor of $a(x)$.

\begin{definition}[$\|f(x)\| , \auxpolg(x) , \auxpol(x)$] \label{Definition_dhg||}
	Let $f(x) \in \Z[x]$. If $f(x) \in \Z$, then we define $\|f(x)\| := 0$ and $\auxpol(x) := \auxpolg(x) := f(x)$. So we assume that $f(x) \not \in \Z$. We then let $\|f(x)\|$ be the maximal $n \in \Z_{\geq 0}$ such that $f(x) \in \Z[x^n]$. Then there is a unique $\auxpolg(x) \in \Z[x]$ such that $f(x) = \auxpolg(t^{\|f(x)\|})$. We then define $\auxpol(x)$ to be the gcd of all closed divisors of $\auxpolg(x)$.
\end{definition}

So $\auxpol(x)$ is a closed divisor of $\auxpolg(x)$ and $f(x) = \auxpolg(x^{\|f(x)\|})$.

\begin{definition}[$\Delta(f(x))$] \label{Definition_Delta}
	We define the map $\mapl{\Delta}{\Z[x]}{\Z[x]}{f(x)}{\auxpol(x)}$. 
\end{definition}

\begin{definition}[Higman-solvable] \label{Definition_HigmanSolvable}
	Let $f(x) \in \Z[x]$. We say that $f(x)$ is \emph{Higman-solvable} if $f(0) \cdot f(1) \neq 0$ and $\Delta^l(f(x)) \in \Z$ for some $l \in \Z_{\geq 0}$. 
\end{definition}

\begin{example} \label{Example_Delta} If $f(x) := (x^4+3x^2+1)(x^2+1)(x+2)$, then $\Delta(f(x)) = (x^4+3x^2+1)(x^2+1)$ and $\Delta^2(f(x)) = 1$. If $f(x) := (x^4-5)(x^2-2)(x+1)$, then $\Delta^1(f(x)) = (x^4-5)(x^2-2)$ and $\Delta^2(f(x)) = x^2-5$ and $\Delta^3(f(x)) = 1$. If $f(x)$ is $ (x^2-2)(x^3-3)$, then $\Delta(f(x)) = f(x)$, so that $f(x)$ is \emph{not} Higman-solvable. If $f(x) := (x^n-1)/(x-1)$, for some composite number $n$, then $\Delta(f(x)) = f(x)$, so that $f(x)$ is \emph{not} Higman-solvable. If $f(x) := x^n-1$, then $f(1) = 0$, so that $f(x)$ is \emph{not} Higman-solvable.
 \end{example}

\begin{definition}[$\HL(f(x))$] \label{Definition_HigmanSolvableLength}
	Let $f(x) \in \Z[x]$ be Higman-solvable. We define the Higman-length $\HL(f(x))$ of $f(x)$ to be the minimal $l \in \Z_{\geq 0}$ such that $\Delta^l(f(x)) \in \Z$.
\end{definition}

\begin{definition}[$\irred(f(x))$] \label{Definition_Irred}
	For $f(x) \in \Z[x] \setminus \{0\}$, we define $\irred(f(x))$ to be the number of irreducible factors of $f(x)$ of positive degree (counted with multiplicity).
\end{definition}

\begin{proposition} \label{Proposition_BoundHigmanLEngth}
	If $f(x) \in \Z[x]$ is Higman-solvable, then $\HL(f(x)) \leq \irred (f(x))$.
\end{proposition}

\begin{proof}
	Note that, for a Higman-solvable polynomial $f(x)$, we have: $\irred(f(x)) = \irred(\Delta(f(x)))$ $\iff$ $f(x) \in \Z$.  We now use induction on $\irred(f(x))$. If $\irred(f(x)) = 0$, then we indeed have $\HL(f(x)) = 0 \leq 0 = \deg(f(x))$. So we assume $\irred(f(x)) > 0$. Since $f(x) \not \in \Z$, we have $1 + \irred(\Delta(f(x))) \leq \irred(f(x))$. The induction hypothesis yields $\HL(\Delta(f(x))) \leq \irred(\Delta(f(x)))$. So also $\HL(f(x)) = 1 + \HL(\Delta(f(x))) \leq 1 + \irred(\Delta(f(x))) \leq \irred(f(x))$.
\end{proof}

\begin{lemma} \label{Proposition_Generate}
	Let $m,n \in \Z_{\geq 1}$ and let $a(x) \in \Z[x] \setminus \{ x , -x , x-1 , -x + 1\}$.
	\begin{enumerate}[\rm (a)]
		\item If $a(x)$ is irreducible, then $a(x)$ is Higman-solvable and $\HL(a(x)) \leq 1$.
		\item If $a(x)$ is Higman-solvable, then $a(x^n)$ is Higman-solvable and $\HL(a(x^n)) = \HL(a(x))$.
		\item Suppose that $a(x)$ is Higman-solvable and monic of degree $m$ and let $C$ be its companion operator. Let $b(x) \in \Z[x]$ be any polynomial of degree at most $n$ that is irreducible and indecomposable and that is coprime to the non-zero, integer polynomial $x \cdot (x-1) \cdot \chi_{C \otimes C}(x) \cdot \chi_{C^1}(x^1) \cdot \chi_{C^2}(x^2) \cdots \chi_{C^{m+n}}(x^{m+n})$. Then also the product $a(x) \cdot b(x)$ is Higman-solvable.
	\end{enumerate}
\end{lemma}

\begin{proof}
	(a) By construction, $\Delta(a(x))$ coincides with the content of $a(x)$. \\ 
	
	(b) Since $\|a(x^n)\| = n \cdot \|a(x)\|$, we have $\Delta(a(x^n)) = \Delta(a(x))$. So the polynomial $a(x^n)$ is Higman-solvable of length $\HL(a(x^n)) = \HL(a(n))$. \\
	
	(c) The eigenvalues of $C$ are precisely the roots of $a(x)$. So the roots of $\chi_{C \otimes C}(x)$ are the products $\lambda \cdot \mu$ of the roots $\lambda,\mu$ of $g(x)$, and the roots of $\chi_{C^i}(x^i)$ are the the product $\omega \cdot \lambda$ of the roots $\omega$ of $x^i - 1$ with the roots $\lambda$ of $a(x)$. Since $b(x)$ is coprime to $x \cdot (x-1)$, we have $b(0) \cdot b(1) \neq 0$, so that also $f(0) \cdot f(1) \neq 0$. Since $b(x)$ is coprime to $\chi_{C \otimes C}(x)$, it suffices to show that every powerful divisor of $f(x)$ of positive degree divides $a(x)$. Let $u(x)$ be a powerful divisor of $f(x)$ of positive degree. Then $\|u(x)\| \leq \deg(f(x)) \leq m + n$. If $\gcd(b(x),u(x)) \neq 1$, then there is a root $\lambda$ of $b(x)$ and a root $\mu$ of $a(x)$ such that $(\lambda / \mu)^{\|u(x)\|}$. In this case, $\gcd(b(x),\chi_{C^{i}}(x^{i})) \neq 1$, for $i := \|u(x)\|$. This contradiction finishes the proof. 
\end{proof}

\begin{remark} \label{Remark_Bound_HL}
	The proof of claim (c) also shows that $\HL(a(x)) \leq \HL(f(x)) \leq \HL(a(x)) + 1$. The upper bound is reached if and only if $a(t)$ is powerful. 
\end{remark}

Let us now prove a more precise version of Proposition \ref{Proposition_Generic}.

\begin{proposition} \label{Proposition_Generic_Detailed}
	Fix $c,d_1,n_1, \ldots, d_c,n_c \in \Z_{\geq 1}$ and set $k := |\{ i | i \leq c-1 \text{ and } n_i \geq 2 \}|$. For each $h \in \Z_{\geq 1}$, we consider the $c$-tuples $(g_1(x),\ldots,g_c(x))$ of monic polynomials in $\Z[x]$ with prescribed degrees $\deg(g_1(x)) = d_1 , \ldots , \deg(g_c(x)) = d_c$ and heights at most $h$. Let $P(h)$ be the probability that the product $f(x) = g_1(x^{n_1 \cdots n_c}) \cdots g_{c-1}(x^{n_{c-1} \cdot n_c}) \cdot g_c(x^{n_c}) $ is Higman-solvable of Higman-length $1+k$. Then $\lim_{h \rightarrow + \infty} P(h) = 1$.
\end{proposition}

\begin{proof}
	Set $d := d_1 + \cdots + d_c$ and let $B_c(h)$ be the number of $c$-tuples $(g_i(x))_{1\leq i \leq c}$ for which $f(x)$ \emph{fails} to satisfy the properties of the proposition. Since $P(h) = 1 - B_c(h) / (2h+1)^d$, we need only show that $B_c(h)$ is $o(h^d)$. We proceed by induction on $c$. Suppose first that $c = 1$. Lemma \ref{Proposition_Generate} (a) and (b) imply that if $f(x) = g_1(x^{n_1})$ is \emph{not} Higman-solvable, then $g_1(x)$ is reducible. Theorem \ref{Theorem_vanderWaerden} implies that the number of reducible $g_1(x)$ is $o(h^{d})$. Since none of the $g_1(x)$ are constant, we conclude that $B_1(h) = o(h^d)$. Now suppose that $c > 1$. We prove the claim for $n_{c-1} > 1$, the other case being a straight-forward variation. For a given $c$-tuple, we define $a(x) := g_1(x^{n_1 \cdots n_{c-1}}) \cdots g_{c-1}(x^{n_{c-1}})$ and $b(x) := g_c(x)$. Lemma \ref{Proposition_Generate} (c) and Remark \ref{Remark_Bound_HL} imply that if $f(x^{1/n_c})$ is \emph{not} Higman-solvable of Higman-length $k$, then $a(x)$ is not Higman-solvable of length $k-1$, or $b(x)$ is reducible or decomposable, or $b(x)$ divides the polynomial defined in the proposition. The induction hypothesis guarantees that there are at most $o(h^{d})$ tuples of the first kind, while Theorem \ref{Theorem_vanderWaerden} guarantees that there are at most $o(h^d)$ of the second kind. For the remaining tuples, it is clear that there are only $o(h^d)$ of the third kind. So the number of $c$-tuples for which $f(x^{1/{n_c}})$ is not Higman-solvable of Higman-length $k$ is $o(h^d)$. Lemma \ref{Proposition_Generate} (b) now implies that $B_c(h) = o(h^d)$.
\end{proof}

\begin{remark}
	By replacing Theorem \ref{Theorem_vanderWaerden} with Bhargava's recent solution \cite{Bhargava} of the van der Waerden conjecture, we can make the growth rates of $B_c(h)$ and $P(h)$ explicit.
\end{remark}

\subsection{More invariants and their properties}

\emph{Let us for the rest of this subsection assume that $f(x) \in \Z[x]$ satisfies $f(0) \neq 0$.}

\begin{definition}[$\auxpolg_{n,i}(x),\auxpoldn_n(x)$] \label{Definition_gni_dn}
	Let $\auxpolg(x)$ be given by $b_0 + b_1 \cdot x + \cdots + b_m \cdot x^m$. For each $n \in \Z_{\geq 2}$ and $i \in \{ 0 , \ldots , n-1\}$, we define the partial sums $\auxpolg_{n,i}(x) := \sum_{j \equiv i \operatorname{mod} n} b_j \cdot x^j.$ We also define $\auxpoldn_n(x) := \gcd(\auxpolg_{n,0}(x),\ldots,\auxpolg_{n,n-1}(x))$. 
\end{definition}

We note that $\auxpolg_{n,i}(x) \in t^i \cdot \Z[x^n]$ and $\auxpoldn_n(x) \in \Z[x^n]$. 

\begin{lemma} \label{Lemma_RhoDef} \label{Lemma_ConstA}
	There exists some $\rho \in \Z_{\geq 1}$ such that, for all $n \in \Z_{\geq 2}$, we have $\rho \cdot \auxpol(x) \in \auxpolg_{n,0}(t) \cdot \Z[x] + \cdots + \auxpolg_{n,n-1}(t) \cdot \Z[x].$
\end{lemma}

\begin{proof}
	Consider $n \geq 2$. Then the $\auxpolg_{n,i}(x)/\auxpoldn_n(x)$ are coprime over $\Q$. So Bezout's theorem gives some $r_n \in \Z \setminus \{0\}$ such that $r_n \cdot \auxpoldn_n(x) \in \auxpolg_{n,0}(t) \cdot \Z[x] + \cdots + \auxpolg_{n,n-1}(t) \cdot \Z[x]$. Now define $\rho := r_2 \cdots r_{m+1} \in \Z \setminus\{0\}$. Since each $\auxpoldn_n(x)$ is powerful, we have $\auxpoldn_n(x) | \auxpol(x)$. For all $n \geq 2$, we then have $\rho \cdot \auxpol(x) \in r_n \cdot \auxpoldn_n(x) \cdot \Z[x] \subseteq \auxpolg_{n,0}(t) \cdot \Z[x] + \cdots + \auxpolg_{n,n-1}(t) \cdot \Z[x]$. 
\end{proof}

The above existence result can be made effective by means of the Euclidean algorithm. The integers $\rho$ form a non-zero principal ideal of $\Z$. This justifies the following definition.

\begin{definition} [$\constA$] \label{Definition_ConstA}
	We define $\constA$ to be the minimal $\rho \in \Z_{\geq 1}$ for which Lemma \ref{Lemma_RhoDef} holds.
\end{definition}

\begin{definition}[$\constB$] \label{Definition_ConstB} \label{Lemma_PowerReductiond^k} If $\auxpol(x)$ is constant, then we define $\constB := \auxpol(x)$. If $\auxpol(x)$ is not a constant, but $\auxpolg(x)/\auxpol(x)$ is a constant, then we define $\constB := \auxpolg(x) / \auxpol(x)$. Else, we define $\constB := \operatorname{Res}(\auxpol(x),\auxpolg(x)/\auxpol(x))$.
\end{definition}

\begin{lemma} \label{Lemma_PowerReductiond^k} \label{Lemma_ConstB} We have $\constB \in \Z \setminus \{0\}$, and 
	for all $k \in \Z_{\geq 1}$, we have $\constB^k \cdot \auxpol(x) \in \auxpol(x)^k \cdot \Z[x] + \auxpolg(x) \cdot \Z[x].$
\end{lemma}

\begin{proof}
	We may assume that $\auxpol(x)$ and $\auxpolg(x)/\auxpol(x)$ are not constant. Since $\auxpol(x)$ is a closed divisor of $\auxpolg(x)$, we have $\constB \in \Z \setminus \{0\}$. We next observe that $\constB^{k-1} \cdot \auxpol(x) \in \auxpol(x) \cdot (\auxpol(x)^{k-1}) \cdot \Z[x] + \auxpol(x) \cdot (\auxpolg(x) / \auxpol(x)) \cdot \Z[x] = \auxpol(x)^k \cdot \Z[x] + \auxpolg(x) \cdot \Z[x]$ for al $k \geq 1$. Then also $\constB^k \cdot \auxpol(x) \in \auxpol(x)^k \cdot \Z[x] + \auxpolg(x) \cdot \Z[x]$ for all $k \geq 1$.
\end{proof}

We use the standard notation $\operatorname{lc}(\auxpol(x))$ for the \emph{leading coefficient} of $\auxpol(x) \in \Z \setminus\{0\}$.

\begin{definition}[$\auxpol^2(x)$] \label{Definition_d2}
	If $\auxpol(x) \in \Z$, then we define $\auxpol^2(x) := \auxpol(x)$. Else, we let $C$ be the companion operator of the monic polynomial ${\operatorname{lc}(\auxpol(x))}^{-1} \cdot \auxpol(x)$. Then we define the polynomial $\auxpol^2(x)$ as the multiple ${\operatorname{lc}(\auxpol(x))}^{k^{2k}} \cdot \chi_{C \otimes C}(x)$ of the characteristic polynomial of the Kronecker square $C \otimes C$ of $C$. 
\end{definition}

We note that $\auxpol^2(x) \in \Z[x] \setminus \{0\}$. 

\begin{lemma} \label{Lemma_D2Roots}
	Let $Z$ be the (possibly empty) root set of $\auxpol(x)$ in an algebraically-closed field $\F$. If $\operatorname{lc}(\auxpol(x)) \not \equiv 0 \mod \operatorname{char}(\F)$, then $Z \cdot Z$ is the root set of $\auxpol^2(x)$ in $\F$.
\end{lemma} 

\begin{proof}
	We assume that $\auxpol(x) \not \in \Z$, since otherwise there is nothing to prove. We see that $Z$ is the root set of ${\operatorname{lc}(\auxpol(x))}^{-1} \cdot \auxpol(x)$. So $Z$ is the set of eigenvalues of the corresponding companion operator $C$. So $Z \cdot Z$ is the set of eigenvalues of the Kronecker square $C \otimes C$. So $Z \cdot Z$ is the set of roots of $\chi_{C \otimes C}(t)$ and therefore of $\auxpol^2(x)$.
\end{proof}

\begin{definition}[$u(x),\constC$] \label{Definition_ConstC}
	We define $u(x) := \gcd(\auxpol^2(x),\auxpolg(x))$. If $\auxpolg(x)/u(x)$ is constant, then we define $\constC := \auxpolg(x)/u(x)$. If $\auxpol^2(x)/u(x)$ is a constant, but $\auxpolg(x)/u(x)$ is not a constant, then we define $\constC := \auxpol^2(x)/u(x)$. Else, we define $\constC := \operatorname{Res}(\auxpol^2(x)/u(x) , \auxpolg(x) / u(x))$.
\end{definition}

\begin{lemma} \label{Lemma_Reductiond2} \label{Lemma_ConstC}
	We have $\constC \in \Z \setminus \{0\}$, and for all $k \in \Z_{\geq 1}$, we have $\constC^k \cdot \auxpol(x)^k \in \auxpol^2(x)^k \cdot \Z[x] + \auxpolg(x) \cdot \Z[x].$
\end{lemma}

\begin{proof}
	We may again assume that $\auxpol^2(x)/u(x)$ and $\auxpolg(x)/u(x)$ are not constant. By construction, $\auxpol^2(x)/u(x) $ and $ \auxpolg(x) / u(x)$ are coprime. So $\constC \in \Z \setminus \{0\}$. For all $k \geq 1$, we get $\constC^k \cdot u(x)^k \in u(x)^k \cdot (\auxpol^2(x)/u(x))^k \cdot \Z[x] + u(x)^k \cdot \auxpolg(x) \cdot \Z[x] \subseteq \auxpol^2(x)^k \cdot \Z[x] + \auxpolg(x) \cdot \Z[x]$. {Since $\auxpol(x)$ is a closed divisor of $\auxpolg(x)$, we see that $u(x) | \auxpol(x)$.} So also $\constC^k \cdot \auxpol(x)^k \in \auxpol^2(x)^k \cdot \Z[x] + \auxpolg(x) \cdot \Z[x]$ for all $k \geq 1$.
\end{proof}

\subsection{The invariant $\HI(f(x))$}

\begin{definition}[$\HI (f(x))$] \label{Definition_MainInvariant} 
	Let $f(x) \in \Z[x]$ be Higman-solvable. If $\HL(f(x)) = 0$, then we define $\HI (f(x)) := f(x)$. Else, we recursively  define
	\begin{eqnarray*}
		\HI(f(x)) &:=& \phantom{\cdot} f(1) \cdot \operatorname{lc}(f(x)) \cdot \constA \cdot \constB \cdot \constC \\
		&& \cdot \discr{f(x)} \cdot \prodant{f(x)} \\
		&& \cdot \HI (\Delta(f(x))).
	\end{eqnarray*}
\end{definition}

\begin{proposition} \label{Proposition_MainInvNonZero}
	If $f(x) \in \Z[x]$ is Higman-solvable, then $\HI (f(x)) \in \Z \setminus \{0\}$.
\end{proposition}  

\begin{proof}
	We use induction on $\HL(f(x))$. If $\HL(f(x)) = 0$, then there is nothing to prove. So we assume $\HL(f(x)) > 0$. Since $\Delta(f(x))$ is again Higman-solvable, the induction hypothesis states that $\HI (\Delta(f(x))) \in \Z \setminus\{0\}$. It now remains to observe that: $f(1) \in \Z \setminus \{0\}$ by assumption, $\operatorname{lc}(\auxpol(x)) \in \Z \setminus \{0\}$ by construction, $\constA \in \Z \setminus \{0\}$ by definition, $\constB \in \Z \setminus \{0\}$ by Lemma \ref{Lemma_PowerReductiond^k}, $\constC \in \Z \setminus \{0\}$ by Lemma \ref{Lemma_Reductiond2}, and $\discr{f(x)} , \prodant{f(x)} \in \Z \setminus \{0\}$ by Proposition $2.1.5$ of \cite{MoensIdentitiesGroups}.
\end{proof}

\section{A reduction from $(G,\alpha,f(x))$ to $(\underline{\operatorname{F}_1}(G),\alpha^{\|f(x)\|},\auxpol(x))$} \label{Section_Aux}

\begin{definition}
	Let $G$ be a group and let $\map{\varphi}{G}{G}$ be a map. We say that $\varphi$ is \emph{nilpotent}, if there exists some $n \in \N$ such that $\varphi^n(x) = 1$ for all $x \in G$. We say that $G$ is a \emph{$\varphi$-group} if $\varphi$ is nilpotent.
\end{definition}

The following property will be used repeatedly (but not always explicitly).

\begin{lemma} \label{Lemma_Concat}
	Suppose that $N$ is normal in $G$ with $\alpha(N) = N$. Let $\map{\overline{\alpha}}{G/N}{G/N}$ be the natural automorphism of $G/N$. Let $f(x)\in \Z[x]$. If $G/N$ is an $f(\overline{\alpha})$-group and $N$ is an $f({\alpha})$-group, then also $G$ is an $f(\alpha)$-group.
\end{lemma}

\begin{proof}
	By assumption, there exist $k_1 , k_2 \in \Z_{\geq 1}$ such that ${f(\alpha)}^{k_1}(G) \subseteq H$ and ${f(\alpha)}^{k_2}(N) \subseteq\{1\}$. So $\varphi^{k_1 + k_2}(G) = \{1\}$.
\end{proof}

In what follows, we will often write slightly inaccurately ``$N$ is an $f(\alpha)$-group'' instead of ``$N$ is an $f(\overline{\alpha})$-group,'' since no confusion can arise. We had already introduced ordered identities of automorphisms. We now consider two more types of identities.

\begin{definition}
	Let $G$ be a group with an automorphism $\alpha$. The polynomial $f(x) \in \Z[x]$ is an \emph{identity} of $\alpha$ if there exist $h_1(x) , \ldots, h_k(x) \in \Z[x]$ with $f(x) = h_1(x) + \cdots + h_k(x)$ and the map $G \longrightarrow G$ sending $g$ to $(h_1(\alpha)(g)) \cdots (h_k(\alpha)(g))$ vanishes identically. We say that $f(x)$ is an \emph{abelian identity} of $\alpha$ if $f(x)$ is an (ordered) identity of all the automorphisms $\map{\overline{\alpha}}{S}{S}$ induced on the abelian, characteristic sections $S$ of $G$.
\end{definition}

It is clear that every ordered identity of $\alpha$ is also an identity of $\alpha$, and that every identity of $\alpha$ is also an abelian identity of $\alpha$. We will need this weakest notion of identity to make our induction work lateron.

\begin{lemma} \label{Lemma_IdealOfAbelianIdentities}
	Let $\alpha$ be an automorphism of a group $G$. Then the abelian identities of $\alpha$ form an ideal of $\Z[x]$.
\end{lemma}

\begin{proof}
	Let $f(x) , g(x) \in \Z[x]$. Suppose first that $f(x), g(x)$ are an abelian identities of $\alpha$ and set $h(x) := f(x) + g(x)$. Then $f(\alpha)$ and $g(\alpha)$ vanish on every characteristic, abelian section $S$ of $G$. For every $s \in S$, we then have $h(\alpha)s = (f(\alpha)s) \cdot (g(\alpha) s) = 1 \cdot 1 = 1$. So also $f(x) + g(x)$ is an abelian identity of $\alpha$. We next suppose that $f(x)$ is an abelian identity of $\alpha$ but that $g(x)$ need not be one. We set $h(x) := f(x) \cdot g(x)$ and we choose a characteristic, abelian section $S$ of $G$. For every $s \in S$, we then have $h(\alpha)s = g(\alpha)( f(\alpha)s ) = g(\alpha)(1) = 1$. So also $f(x) \cdot g(x)$ is an abelian identity of $\alpha$. This finishes the proof.
\end{proof}

\begin{lemma} \label{Lemma_AbelianPowerIdentity}
	Let $G$ be a finite, solvable group of derived length $k$ with automorphism $\alpha$. If $f(x) \in \Z[x]$ is an abelian identity of $\alpha$, then $G$ is an $f(\alpha)$-group and $f(x)^k$ is an identity of $\alpha$.
\end{lemma}

\begin{proof}
	We see that $f(\alpha)$ vanishes on each of the $k$ abelian factors of the derived series of $G$. So then $f(\alpha)^k(G) = \{1\}$ and $f(x)^k$ is an identity of $\alpha$.
\end{proof}

\begin{proposition} \label{Proposition_QGroup}
	Let $f(x) \in \Z[x]$ satisfy $f(0) \neq 0$ and let $q$ be a prime satisfying $\gcd(q,\operatorname{lc}(f(x)) \cdot \constC) = 1$. Let $Q$ be a finite $q$-group and let $\beta$ be an automorphism. Suppose that $\auxpolg(x)$ is an abelian identity of $\beta$ and suppose that the Frattini-quotient $Q / \Phi(Q)$ is an $\auxpol(\beta)$-group. Then also $Q$ is an $\auxpol(\beta)$-group.
\end{proposition}

\begin{proof}
	The lower central series $(\Gamma_i(Q))_{i \geq 0}$ of $Q$ naturally gives rise to a finite, graded Lie ring $L := L_1 \oplus L_2 \oplus \cdots \oplus L_c$ of class $c = \operatorname{c}(L)  = \operatorname{c}(Q)$, where $L_i :=  \Gamma_i(Q) / \Gamma_{i+1}(Q)$. Let $\map{\gamma}{L}{L}$ be the Lie ring automorphism that is naturally induced on $L$ by $\beta$. Then $\auxpolg(\gamma)$ vanishes on $L$. Since $\Phi(Q) = Q^q \cdot [Q,Q]$, we also have $\auxpol(\gamma)^{k_0}(L_1) \subseteq q \cdot L_1$, for some sufficiently large $k_0 \in \Z_{\geq 1}$. \\ 
	
	Let us prove that $(L,+)$ is an $\auxpol(\gamma)$-group. We first consider the special case: $q \cdot L = \{0_L\}$. Then $L$ is a Lie algebra over the prime field $\F_q$ and $\auxpol(\gamma)^{k_0}(L_1) = \{0_L\}$. After extending the scalars, we may further assume that the ground field $\F$ of $L$ is algebraically-closed. We proceed by induction on the class $c$ of $L$. If $c \leq 1$, then $L = L_1$ is an $\auxpol(\gamma)$-group by assumption. So we assume $c > 1$. The induction hypothesis yields some $k_1 \in \Z_{\geq 1}$ such that
		$
			\auxpol(\gamma)^{k_1}(L_1 \oplus \cdots \oplus L_{c-1}) = \{0_L\}. \label{Eq_dt}
		$ 
		So the generalized eigenvalues of $\gamma$ on $L_1$ and $L_{c-1}$ are roots of $\auxpol(x)$. 
		The generalized eigenvalues of $\gamma$ on $L_{c} = [L_1,L_{c-1}]$ will therefore be of the form $\lambda \cdot \mu$ with $\auxpol(\lambda) = \auxpol(\mu) = 0$. Since $(q,\operatorname{lc}(f(x))) = 1$, we may apply Lemma \ref{Lemma_D2Roots} to obtain $\auxpol^2(\gamma)^{k_2}(L_{c}) \subseteq \auxpol^2(\gamma)^{k_2}([L_1,L_{c-1}]) = \{0_L\}, \label{Eq_d2t}$ for some sufficiently large $k_2 \in \Z_{\geq 1}$. But, by assumption, we also have $ \auxpolg(\gamma)(L_{c}) = \{0_L\}. \label{Eq_gt}$ So $\auxpol^2(x)^{k_2}$ and $\auxpolg(x)$ are both identities of $\gamma$ on $L_{c}$. {Lemma} \ref{Lemma_Reductiond2} shows that the ideal of $\Z[x]$ that is generated by $\auxpol^2(x)^{k_2}$ and $\auxpolg(x)$ contains the polynomial $\constC^{k_2} \cdot \auxpol(x)^{k_2}$. Lemma \ref{Lemma_IdealOfAbelianIdentities} then implies that $\constC^{k_2} \cdot \auxpol(\gamma)^{k_2}(L_{c}) = \{0_L\}$. Since $(q,\constC)  = 1$, also $\auxpol(\gamma)^{k_3}(L_{c}) = \{0_L\}$. Set $k_3 := \max\{k_1 , k_2\}$. Then we conclude that $\auxpol(\gamma)^{k_3}(L) \subseteq \auxpol(\gamma)^{k_3}(L_1 \oplus \cdots \oplus L_{c-1}) + \auxpol(\gamma)^{k_3}(L_{c}) = \{0_L\}$. We now consider the general case: $q^e \cdot L = \{0_L\}$, where $e \in \Z_{\geq 1}$. Then $q \cdot L$ is a characteristic ideal of $L$ and the additive exponent of the quotient $L / (q \cdot L)$ is $q$. So we can apply the above argument to $L / (q \cdot L)$ in order to conclude that $\auxpol(\gamma)^{k_3}(L) \subseteq q \cdot L$, for some $k_3 \in \Z_{\geq 1}$. But then $\auxpol(\gamma)^{e \cdot k_3}(L) \subseteq q^{e} \cdot L = \{0_L\}$. \\

	We can finally derive that $Q$ is an $\auxpol(\gamma)$-group. A simple induction on $i \in \Z_{\geq 1}$ shows that $\auxpol(\beta)^{i \cdot e \cdot k_3}(Q) \subseteq \Gamma_{i+1}(Q)$. So  $\auxpol(\beta)^{c \cdot e \cdot k_3}(Q) \subseteq \Gamma_{c+1}(Q) = \{1_Q\}$.
\end{proof}

\begin{proposition} \label{Theorem_Aux}
	Let $G$ be a finite, solvable group with an automorphism $\alpha$ and let $f(x)$ be an abelian identity of $\alpha$. Suppose that $f(0) \neq 0$ and $\gcd(|G|,f(1) \cdot \operatorname{lc}(f(x)) \cdot \constA \cdot \constC) = 1$. Then the automorphism $\map{{\alpha}^{\|f(x)\|}}{\underline{\operatorname{F}_1}(G)}{\underline{\operatorname{F}_1}(G)}$ is fixed-point-free and $\underline{\operatorname{F}_1}(G)$ is an $\auxpol({\alpha}^{\|f(x)\|})$-group. 
\end{proposition}

The first statement is easy to prove. Let us use the abbreviation $\beta := \alpha^{\|f(x)\|}$. Then $\auxpolg(x)$ is an abelian identity of $\beta$. Since $G$ is solvable, we may apply Lemma \ref{Lemma_AbelianPowerIdentity}. So some natural power of $\auxpolg(x)$, say $\auxpolg(x)^{l_0}$, is an identity of $\beta$. Any fixed-point $x$ of $\beta$ therefore satisfies $x^{\auxpolg(1)^{l_0}}$. Since $\gcd(|G|,\auxpolg(1)) = 1$, we conclude that $x = 1$. So $\beta$ is a fixed-point-free automorphism of $G$. The second statement is more difficult to prove. We proceed by induction on $|G|$. If $|G| = 1$, then $\underline{\operatorname{F}_1}(G) = \{1\}$, so that there is nothing to prove. So we assume that $|G| > 1$. We follow the strategy of Higman in \cite{HigmanGroupsAndRings}.

\begin{claim}
	We may assume that $G$ has Fitting height exactly $2$.
\end{claim}

\begin{proof}
	If the Fitting height is at most $1$, then there is nothing left to prove. So suppose $\underline{\operatorname{F}_2}(G) \neq \{1\}$. Then $\underline{\operatorname{F}_1}(G)$ is an proper, characteristic section of $G$. By the induction hypothesis, we have that $ \underline{\operatorname{F}_1}(\underline{\operatorname{F}_1}(G)) = \underline{\operatorname{F}_2}(G)$ is a $\auxpol(\beta)$-group. Similarly, $G / \underline{\operatorname{F}_2}(G)$ is a proper, characteristic section of $G$, so that $\underline{\operatorname{F}_1}(G / \underline{\operatorname{F}_2}(G)) = \underline{\operatorname{F}_1}(G) / \underline{\operatorname{F}_2}(G)$ is an $\auxpol(\beta)$-group. But then all of $\underline{\operatorname{F}_1}(G)$ is a $\auxpol(\beta)$-subgroup. \let\qed\relax
\end{proof}

\begin{claim}
	We may assume that the first upper Fitting-subgroup $\overline{\operatorname{F}_1}(G)$ of $G$ is a $q$-group, for some prime $q$.
\end{claim}

\begin{proof}
	Otherwise, the nilpotent group $\overline{\operatorname{F}_1}(G)$ is the direct product of two proper, non-trivial, characteristic subgroups $A$ and $B$. Then the induction hypothesis implies that $\underline{\operatorname{F}_1}(G/A)$ and $\underline{\operatorname{F}_1}(G/B)$ are $\auxpol(\beta)$-groups. By definition, there then exist $k_1,k_2 \in \Z_{\geq 1}$ such that $\auxpol(\beta)^{k_1}(\underline{\operatorname{F}_1}(G)) \subseteq A$ and $\auxpol(\beta)^{k_2}(\underline{\operatorname{F}_1}(G)) \subseteq B$. Since $A \cap B = \{1\}$, we conclude that also $\underline{\operatorname{F}_1}(G)$ is a $\auxpol(\beta)$-group. \let\qed\relax
\end{proof}

\begin{claim}
	We may assume that the quotient $G/\overline{\operatorname{F}_1}(G)$ is an elementary-abelian $p$-group, for some prime $p \neq q$.
\end{claim}

\begin{proof}
	Let $E/\overline{\operatorname{F}_1}(G)$ be a proper, characteristic subgroup of $G/\overline{\operatorname{F}_1}(G)$ and lift it to a proper characteristic subgroup $E$ of $G$ containing $\overline{\operatorname{F}_1}(G)$. The induction hypothesis forces $\underline{\operatorname{F}_1}(E)$ to be an $\auxpol(\beta)$-group. If $\underline{\operatorname{F}_1}(E) \neq \{1\}$, then the induction hypothesis forces $\underline{\operatorname{F}_1}(G/ \underline{\operatorname{F}_1}(E)) = \underline{\operatorname{F}_1}(G) / \underline{\operatorname{F}_1}(E)$ to be a $\auxpol(\beta)$-group. But then also $\underline{\operatorname{F}_1}(G)$ is a $\auxpol(\beta)$-group and our proof is done. So we may suppose that $\underline{\operatorname{F}_1}(E) = \{1\}$. This means that the normal subgroup $E$ is nilpotent and contained in the maximal nilpotent normal subgroup $\overline{\operatorname{F}_1}(G)$ of $G$. So $E / \overline{\operatorname{F}_1}(G)$ is the trivial group. We conclude that $G / \overline{\operatorname{F}_1}(G)$ is characteristically-simple. Since $G$ is solvable, the quotient is an elementary-abelian $p$-group, for some prime $p$. Since $G$ has Fitting height exactly $2$, it is not nilpotent, and $p \neq q$. \let\qed\relax
\end{proof}

\begin{claim}
	We may assume that $\underline{\operatorname{F}_1}(G) = \overline{\operatorname{F}_1}(G).$
\end{claim}

\begin{proof}
	Otherwise, the Sylow-$p$ subgroup of the nilpotent group $G/\underline{\operatorname{F}_1}(G)$ is proper, non-trivial, and characteristic so that it lifts to a proper, characteristic subgroup $E$ of $G$ that is \emph{not} a $q$-group. The induction hypothesis implies that $\underline{\operatorname{F}_1}(E)$ is an $\auxpol(\beta)$-group. If it is non-trivial, then the induction hypothesis also claims that $\underline{\operatorname{F}_1}(G/\underline{\operatorname{F}_1}(E)) = \underline{\operatorname{F}_1}(G) / \underline{\operatorname{F}_1}(E)$ is a $\auxpol(\beta)$-group. But in this case all of $\underline{\operatorname{F}_1}(G)$ is a $\auxpol(\beta)$-group, and we are done. The other case cannot occur, since it would imply that $E \subseteq\overline{\operatorname{F}_1}(G)$ is a $q$-group. \let\qed\relax
\end{proof}


\begin{claim}
	We may assume that $F := \underline{\operatorname{F}_1}(G) = \overline{\operatorname{F}_1}(G)$ is an elementary-abelian $q$-group. 
\end{claim}

\begin{proof}
	Otherwise, the Frattini-subgroup $\Phi(F) $ of $F$ is non-trivial. The induction hypothesis then forces $\underline{\operatorname{F}_1}(G/\Phi(F)) = F/\Phi(F)$ to be an $\auxpol(\beta)$-group. Proposition \ref{Proposition_QGroup} then guarantees that all of $F$ is an $\auxpol(\beta)$-group, and we are done. \let\qed\relax
\end{proof}

Lemma \ref{Lemma_UniformHallQuotient} provides a Sylow-$p$ subgroup $P$ of $G$ satisfying $\beta(P) = P$. Then $G = F \rtimes P$ and $G \rtimes \langle \beta \rangle = F \rtimes ( P \rtimes \langle \beta \rangle )$. So $P \rtimes \langle \beta \rangle$ acts on $F$ via conjugation within $G \rtimes \langle \beta \rangle$. We now change our perspective and notation. We identify the abelian section $F$ of $G$ with the additive group of a non-trivial, finite-dimensional vector space $V$ over the prime field $\F_q$. Then $P \rtimes \langle \beta \rangle$ naturally acts on $V$ via the homomorphism $\mapl{\overline{\cdot}}{P \rtimes \langle \beta \rangle}{\operatorname{GL}(V)}{A}{\overline{A}}$. 

\begin{claim} \label{Claim_3Conditions}
	$\overline{P}$ acts fixed-point-freely on $V$, $\overline{\beta}$ normalizes $\overline{P}$ and acts fixed-point-freely on $\overline{P}$, and $\auxpolg(\overline{\beta})(V) = \{0_V\}$.
\end{claim}

\begin{proof}
	Claim $1$ follows from the fact that $F = \overline{\operatorname{F}_1}(G)$ is self-centralizing. Consider claim $2$. Since $P$ is normal in $P \rtimes \langle \beta \rangle$, its image $\overline{P}$ is normal in $\overline{P \rtimes \langle \beta \rangle}$. Now suppose that $A \in P$ and that $\overline{\beta}^{-1} \circ \overline{A} \circ \overline{\beta}$ acts trivially on $V$. Then the element $\beta^{-1} \cdot A \cdot \beta$ of $P$ acts trivially on $F$. Since $F$ is self-centralizing, we conclude that $\beta^{-1} \cdot A \cdot \beta = 1_P$. So also $\overline{\beta}^{-1} \circ \overline{A} \circ \overline{\beta} = \id_V$. Claim $3$ is simply a change of notation. \let\qed\relax 
\end{proof}

It now suffices to verify that $\auxpol(\overline{\beta})(V) = \{0_V\}$. Let us do this in the remainder of the proof. We may assume that the base field $\F$ of $V$ is algebraically-closed. Since $\gcd(q,|\overline{P}|) = 1$, we can simultaneously diagonalize the elements of $\overline{P}$. So $V = \bigoplus_{\chi} V_{\chi},$ where $\chi$ runs over the characters of $\overline{P}$ with \emph{non-zero} character space $V_\chi$. We recall that, for every $\overline{A} \in \overline{P}$ and every $v \in V_\chi$, we have $\overline{A}(v) = \chi(\overline{A}) \cdot v$. One can then verify that the group $\langle \overline{\beta}\rangle$ naturally acts on the set of character spaces by permutations. 

\begin{claim}
	For every character space $V_\chi$, we have $\overline{\beta}(V_\chi) \neq V_\chi$. \label{Claim_OrderOrbit}
\end{claim}

\begin{proof}
	Suppose otherwise: $\overline{\beta}(V_\chi) = V_\chi$. For every $\overline{A} \in \overline{P}$ and $v \in V_\chi$, we then have $(\overline{\beta}^{-1} \circ \overline{A} \circ \overline{\beta})(v) = \overline{A}(v)$. So every element of $V_\chi$ is a common fix-point of $\overline{A}^{-1} \cdot (\overline{\beta}^{-1} \cdot \overline{A} \cdot \overline{\beta})$. Since $\overline{\beta}$ acts fixed-point-freely on $\overline{P}$, Lemma \ref{Lemma_UniformHallQuotient} shows that every element of $\overline{P}$ is of the form $\overline{A}^{-1} \cdot (\overline{\beta}^{-1} \cdot \overline{A} \cdot \overline{\beta})$. So every element of $V_\chi$ is a common fix-point of $\overline{P}$. Claim \ref{Claim_3Conditions} then forces $V_\chi = \{0_V\}$. But every character space is non-zero \emph{by definition}. \let\qed\relax
\end{proof}

\begin{claim}
	We have $\auxpol(\overline{\beta})(V) = \{0_V\}$.
\end{claim}

\begin{proof}
	Let $V_\chi$ be an arbitrary character space and consider its orbit $V_\chi , \overline{\beta}(V_\chi) =: V_{\chi_1}$, $\ldots$ , $\overline{\beta}^{n-1}(V_\chi) =: V_{\chi_{n-1}}$ under the action of $\langle \overline{\beta} \rangle$, where $n$ is the minimal element of $\Z_{\geq 1}$ such that $\overline{\beta}(V_\chi) = V_{\chi}$. Claim \ref{Claim_OrderOrbit} shows that $n \geq 2$. For every $v \in V_\chi$, we then have  
	$0_V = \auxpolg(\overline{\beta})(v) = \auxpolg_{n,0}(\overline{\beta})(v) + \cdots + \auxpolg_{n,n-1}(\overline{\beta})(v).$ 
	Since the $\auxpolg_{n,i}(\overline{\beta})(v)$ belong to linearly independent spaces $V_{\chi_i}$, each term of this sum vanishes. By definition (cf. Lemma \ref{Lemma_RhoDef}), there exist polynomials $a_0(x) , \ldots , a_{n-1}(x) \in \Z[x]$ such that $\constA \cdot \auxpol(x) = a_0(x) \cdot \auxpolg_{n,0}(x) + \cdots + a_{n-1}(x) \cdot \auxpolg_{n,n-1}(x)$. By evaluating the latter in $\overline{\beta}$ and then in $v$, we obtain
	\begin{eqnarray*}
		\constA \cdot \auxpol(\overline{\beta}) (v) &=& a_0(\overline{\beta}) \circ \auxpolg_{n,n-1}(\overline{\beta})(v) + \cdots + a_{n-1}(\overline{\beta}) \circ \auxpolg_{n,n-1}(\overline{\beta}) (v) \\
		&=& a_0(\overline{\beta}) (0_V) + \cdots + a_{n-1}(\overline{\beta})(0_V) \\
		&=& 0_V.
	\end{eqnarray*}
	Since $\gcd(\constA,q) | \gcd(\HI (f(x)),|G|) = 1$, we also have $\auxpol(\overline{\beta})(v) = 0_V$. Since $v$ was chosen arbitrarily in $V_\chi$, we even have $\auxpol(\overline{\beta})(V_\chi) = \{0_V\}$. Since the character space $V_\chi$ was chosen arbitrarily, we conclude that $\auxpol(\overline{\beta})(V) = \{0_V\}$. 
\end{proof}

This finishes the proof of the proposition. We can easily ``upgrade'' this result.

\begin{theorem} \label{Proposition_Upgrade}
	Let $G$ be a finite, solvable group with a fixed-point-free automorphism $\map{\alpha}{G}{G}$ and let $f(x)$ be an abelian identity of $\alpha$. Suppose that $f(0) \neq 0$ and $\gcd(|G|, f(1) \cdot  \operatorname{lc}(f(x)) \cdot \constA \cdot \constB \cdot \constC) = 1$. Then $\auxpol(x)$ 
	is an abelian identity of the fixed-point-free automorphism $\map{\alpha^{\|f(x)\|}}{\underline{\operatorname{F}_{1}}(G)}{\underline{\operatorname{F}_{1}}(G)}$. 
\end{theorem}

\begin{proof}
	Proposition \ref{Theorem_Aux} establishes that $\map{\beta := \alpha^{\|f(x)\|}}{\underline{\operatorname{F}_1}(G)}{\underline{\operatorname{F}_1}(G)}$ is fixed-point-free and that $\underline{\operatorname{F}_1}(G)$ is an $\auxpol(\beta)$-group. Lemma \ref{Lemma_AbelianPowerIdentity} then states that ${\auxpol(x)}^{k}$ is an identity of $\beta$,
	for some $k \in \Z_{\geq 1}$. Now let $J$ be the ideal of abelian identities of $\beta$ (cf. Lemma \ref{Lemma_IdealOfAbelianIdentities}). Let $S$ be any characteristic section of $\underline{\operatorname{F}_1}(G)$ that is abelian and let $\map{\beta_S}{S}{S}$ be the induced automorphism. Then ${\auxpol(x)}^{k} \in J$. By construction, $\auxpolg(x) \in J$. According to {Lemma} \ref{Lemma_PowerReductiond^k}, we  have $\constB^{k} \cdot {\auxpol(x)} \in J$. Trivially, we have $|G| \in J$. Since $\gcd(|G|,\constB) = 1$, Euclid's theorem gives $a,b \in \Z$ such that $a \cdot |G| + b \cdot {\constB}^{k} = 1$. But then also ${\auxpol(x)} = 1 \cdot {\auxpol(x)} = (a \cdot |G| + b \cdot \constB^{k}) \cdot {\auxpol(x)} = a \cdot |G| \cdot \auxpol(x) + b \cdot \constB^{k} \cdot \auxpol(x) \in J$. 
\end{proof}

\section{Proof of the main theorem} \label{Section_Main}

\begin{theorem} \label{Theorem_Recursion}
	Let $G$ be a finite, solvable group with a fixed-point-free automorphism $\map{\alpha}{G}{G}$ and let $f(x)$ be an abelian identity of $\alpha$. Suppose that $f(x)$ is Higman-solvable and $\gcd(|G|,\HI (f(x))) = 1$. For each $i \in \Z_{\geq 0}$, $\Delta^{i+1}(f(x))$ is an abelian identity of the fixed-point-free automorphism $\map{\alpha^{\|f(x)\| \cdots \|\Delta^i(f(x))\|}}{\underline{\operatorname{F}_{i+1}}(G)}{\underline{\operatorname{F}_{i+1}}(G)}$. 
\end{theorem}

\begin{proof}
	We first apply Theorem \ref{Proposition_Upgrade}. We then proceed by induction on $i \in \Z_{\geq 0}$. If $i = 0$, there is nothing left to prove. So we assume $i > 0$. By definition, $\HI (\Delta(f(x)))$ divides $\HI (f(x))$. So $\gcd(|\underline{\operatorname{F}_1}(G)|,\HI (\Delta(f(x)))) $ divides $ \gcd(|G|,\HI (f(x)))$ and is therefore $1$. The induction hypothesis then states that $\Delta^{i}(\Delta(f(x))) = \Delta^{i+1}(f(x))$ is an abelian identity of the fixed-point-free automorphism $(\alpha^{\|f(x)\|})^{\|\Delta(f(x))\| \cdots \|\Delta^{i-1}(\Delta(f(x)))\|} = \alpha^{\|f(x)\| \cdots \|\Delta^{i}(f(x))\|}$ on the subgroup $\underline{\operatorname{F}_i}(\underline{\operatorname{F}_1}(G)) = \underline{\operatorname{F}_{i+1}}(G)$ of $G$. 
\end{proof}

\begin{lemma} \label{Lemma_Torsion}
	Let $G$ be a finite, solvable group and let $\map{\alpha}{G}{G}$ be an automorphism. If $m \in \Z \setminus \{0\}$ is an abelian identity of $\alpha$ and $\gcd(|G|,m) = 1$, then $G = \{1\}$.
\end{lemma}

\begin{proof}
	Lemma \ref{Lemma_AbelianPowerIdentity} states that $G^{m^k} = \{1\}$, for some $k \in \Z_{\geq 0}$. So $G = \{1\}$. 
\end{proof}

We now prove a strong form of Theorem \ref{Theorem_Main} by replacing ordered identities with identities, by removing the condition on the torsion of $|G|$, and by obtaining a sharper bound on the Fitting height of $G$. 

\begin{theorem} \label{Theorem_DetailedMain}
	Let $G$ be a finite group with a fixed-point-free automorphism $\alpha$ and let $f(x)$ be an identity of $\alpha$ that is Higman-solvable. 
	\begin{enumerate}[\rm (a)]
		\item Then $G$ is solvable and the product $A \cdot B$ of two subgroups $A$ and $B$ of coprime order, 
		such that the Fitting height $h(A)$ of $A$ satisfies $$h(A) \leq \HL(f(x)) \leq \irred(f(x)) \leq \deg(f(x)),$$ and such that $|B|$ divides a natural power of the non-zero integer $\HI (f(x))$. 
		\item Suppose, moreover, that the roots of $f(x)$ form an arithmetically-free subset $X$ of the group $(\overline{\Q}^\times,\cdot)$. Then the derived length $\operatorname{dl}(A)$ of $A$ satisfies $$
		\operatorname{dl}(A) \leq \HL(f(x)) \cdot \operatorname{H}(X,\overline{\Q}^\times) \leq \deg(f(x))^{2^{\deg(f(x))}+1} .
		$$
	\end{enumerate}
\end{theorem}

Note that if $\gcd(|G|,\HI(f(x))) = 1$, then $B = \{1\}$ and $G = A$. So we do indeed recover Theorem \ref{Theorem_Main}.

\begin{proof}
	(a) The group is solvable by Rowley's Theorem \ref{Theorem_Rowley}. According to Lemma \ref{Lemma_UniformHallQuotient}, we can find a Hall-$\HI (f(x))$ subgroup $B$ of $G$ and a Hall-$\HI (f(x))'$-subgroup $A$ of $G$ satisfying $\alpha(A) = A$. Then $G = A \cdot B$ and we consider the automorphism $\map{\alpha_A}{A}{A}$ obtained by restriction. Then $f(x)$ is an abelian identity of $\alpha_A$. Let $l := \HL(f(x))$. If $l = 0$, then $f(x) = \HI (f(x))$ is a non-zero constant, so that we need only apply Lemma \ref{Lemma_Torsion} to conclude that $A = \{1_A\}$. So we assume that $l > 0$. Theorem \ref{Theorem_Recursion} then shows that the non-zero constant polynomial $\Delta^{l}(f(x))$ is an abelian identity of some automorphism of $\underline{\operatorname{F}_{l}}(A)$. Since $\Delta^{l}(f(x))$ divides $\HI (f(x))$, we need only apply Lemma \ref{Lemma_Torsion} to conclude that $\underline{\operatorname{F}_{l}}(A) = \{1\}$. Proposition \ref{Proposition_BoundHigmanLEngth} further states that $l \leq \irred(f(x)) \leq \deg(f(x))$. \\
	
	(b) We assume that $l := \HL(f(x)) > 0$, since otherwise there is nothing left to prove. Let $S_i := \underline{\operatorname{F}_i}(A) / \underline{\operatorname{F}_{i+1}}(A)$ be the $i$'th factor of the lower Fitting series of $A$. We then have the coarse bound $\operatorname{dl}(A) \leq \operatorname{dl}(S_0) + \cdots + \operatorname{dl}(S_{l-1}) \leq \operatorname{c}(S_0) + \cdots + \operatorname{c}(S_{l-1})$. Proposition \ref{Proposition_BoundHigmanLEngth} shows that $l \leq \deg(f(x))$. So it suffices to show that $\operatorname{c}(S_i) \leq \operatorname{H}(X,\overline{\Q}^\times) \leq |X|^{2^{|X|}}$, for each $i \geq 0$. Let $L$ be the Lie ring that naturally corresponds with the lower central series of some such $S_i$ and let $\map{\beta}{L}{L}$ be the induced Lie ring automorphism. Then $f(x)$ is an identity of $\beta$. Since $\gcd(|L|,\discr{f(x)} \cdot \prodant{f(x)}) | \gcd(|A|, \HI(f(x))) = 1$, we may use Theorem \ref{Theorem_Embedding} to embed $\discr{f(x)} \cdot L$ into a Lie ring $K$ that is graded by $\overline{\Q}^\times$ and supported by $X$. Theorem \ref{Theorem_AF} then states that $K$ is nilpotent of class $\operatorname{c}(K) \leq \operatorname{H}(X,\overline{\Q}^\times)$. Theorem \ref{Theorem_BoundField} further shows that $\operatorname{H}(X,\overline{\Q}^\times) \leq |X|^{2^{|X|}}$. By combining these observations, we obtain $\operatorname{c}(S_i) = \operatorname{c}(L) = \operatorname{c}(\discr{f(x)} \cdot L) \leq \operatorname{c}(K) \leq \operatorname{H}(X,\overline{\Q}^\times) \leq |X|^{2^{|X|}}$. 
\end{proof}


\begin{remark} \label{Remark_Auto}
	Every monic, Higman-solvable polynomial $f(x)$ of positive degree is an ordered identity of a fixed-point-free automorphism of a finite, non-trivial group.
\end{remark}

Indeed, let $p$ be any prime not dividing the non-zero integer $f(0) \cdot f(1)$. Then the companion operator $\alpha$ of $f(x)$ defines a fixed-point-fee automorphism of the elementary-abelian $p$-group of rank $\deg(f(x))$  and $f(x)$ is an ordered identity of $\alpha$. We refer to {\cite[Section~3]{MoensIdentitiesGroups}} for more interesting constructions.

\begin{remark} \label{Remark_Ident}
	Every fixed-point-free automorphism $\alpha$ of a finite, non-trivial group $G$ has a monic, ordered identity that is Higman-solvable.
\end{remark}

Indeed, the polynomial $f(x) := -1 + |G| \cdot x^{|\alpha|-1} + x^{|\alpha|}$ is an ordered identity of $\alpha$. By Perron's criterion, this $f(x)$ is irreducible and therefore Higman-solvable. We refer to {\cite[Section~3]{MoensIdentitiesGroups}} for more interesting constructions.

\begin{remark}
	Suppose that $f(x) \in \Z[x]$ is not divisible by $x$ or by any cyclotomic polynomial. Then its root set is arithmetically-free.
\end{remark}

Indeed, for every pair of roots $(\lambda,\mu)$, the arithmetic progression $\lambda, \lambda \cdot \mu , \lambda \cdot \mu^2 , \ldots $ contains infinitely-many distinct elements and is therefore not contained in the root set.

\section{Examples} \label{Secition_RemarksExamples}


\begin{example} \label{Ex_I1}
	Let $G$ be a finite group with a f.p.f. automorphism $\map{\alpha}{G}{G}$, satisfying $ \{ g^2 \cdot \alpha(g) \cdot \alpha^2(g)^8 \cdot \alpha^3(g)^4 \cdot \alpha^4(g)^8 \cdot \alpha^5(g)^4 \cdot \alpha^6(g)^2 \cdot \alpha^7(g) | g \in G\}  = \{1\}$. Then $G$ is the product of a metabelian subgroup $A$ and a Hall-$(2 \cdot 3 \cdot 5)$ subgroup $B$.
\end{example}

\begin{proof}
	We see that $f(x) := (x^4+3x^2+1)(x^2+1)(x+2)$ is an identity of the automorphism. Example \ref{Example_Delta} shows that $f(x)$ is Higman-solvable with $\HL(f(x)) = 2$. One can verify that $\HI(f(x))$ divides a natural power of $2 \cdot 3 \cdot 5$. Moreover, the root set $X$ of $f(x)$ is product-free, so that $\operatorname{H}(X,\overline{\Q}^\times) = 1$. So we need only apply Theorem~\ref{Theorem_DetailedMain}. 
\end{proof}

\begin{example} \label{Ex_I2}
	Let $G$ be a finite group with a f.p.f. automorphism $\map{\alpha}{G}{G}$, satisfying $ \{g^{10} \cdot \alpha^5(g)^{-2}  \cdot \alpha^2(g)^{-5}  \cdot  \alpha(g)^{10}  \cdot  \alpha^6(g) \cdot  \alpha^4(g)^{-2}  \cdot  \alpha^3(g)^{-5}  \cdot  \alpha^7(g) | g \in G \} = \{1\}$. Then $G$ is the product of a metabelian-by-nilpotent subgroup $A$ and a Hall-$(2 \cdot 3 \cdot 5 \cdot 7 \cdot 11 \cdot 19)$ subgroup $B$.
\end{example}

\begin{proof}
	We see that $f(x) := (x^4-5)(x^2-2)(x+1)$ is an identity of the automorphism. Example \ref{Example_Delta} shows that $f(x)$ is Higman-solvable with $\HL(f(x)) = 3$. One can verify that $\HI (f(x))$ divides a natural power of $2 \cdot 3 \cdot 5 \cdot 7 \cdot 11 \cdot 19$. So we may apply Theorem~\ref{Theorem_DetailedMain} and conclude that $B$ is solvable with Fitting height at most $3$. We next observe that the roots of $f(x)$ \emph{do not form an arithmetically-free subset} of $(\overline{\Q}^\times,\cdot)$. But we know, from Theorem \ref{Theorem_Recursion}, that $\Delta^1(f(x)) = (x^4 - 5)(x^2 - 2)$ is an abelian identity of the f.p.f. automorphism $\map{\alpha}{\underline{\operatorname{F}_1}(B)}{\underline{\operatorname{F}_1}(B)}$. Since $\Delta(f(x))$ has a product-free set of roots, say $Y$, Theorem~\ref{Theorem_DetailedMain} states that $\operatorname{dl}(\underline{\operatorname{F}_1}(B)) \leq \HL(\Delta(f(x))) \cdot \operatorname{H}(Y,\overline{\Q}^\times) \leq (3 - 1) \cdot 1 = 2$. A slightly more complicated argument allows us to show that $B$ is abelian-by-nilpotent. 
\end{proof}

\begin{example}
	Let $G$ be a finite group with f.p.f. automorphism $\map{\alpha}{G}{G}$, satisfying $\{\alpha^3(g)^{-1} \cdot g^3 \cdot \alpha^3(g)^{-1} \cdot g^3 \cdot \alpha^2(g)^{-3} \cdot \alpha^5(g) | g \in G \} = \{1\}$. Then $G$ is the product of an abelian subgroup $A$ and a Hall-$(2 \cdot 3 \cdot 5 \cdot 7 \cdot 11 \cdot 17 \cdot 73)$ subgroup $B$.
\end{example}

\begin{proof}
	We note that $h(x) := (x^2 -2)(x^3-3)$ is an identity of $\alpha$. Example \ref{Example_Delta} shows that it is \emph{not Higman-solvable}. But, since $h(x)$ has {no roots of finite order}, it divides some Higman-solvable polynomial $f(x)$ with $\HL(f(x)) \leq 1$. This $f(x)$ is then an identity of $\alpha$, according to Lemma \ref{Lemma_IdealOfIdentities}, and it satisfies the assumptions of Theorem~\ref{Theorem_DetailedMain}. It now remains to observe that $f(x) := (x^6-2^3)(x^6-3^2)$ is such a polynomial with a product-free set of roots and that $\HI(f(x))$ divides a natural power of $2 \cdot 3 \cdot 5 \cdot 7 \cdot 11 \cdot 17 \cdot 73$.
\end{proof}


\bibliographystyle{plain}
\bibliography{MySpecialBibliography}

\textsc{Faculty of Mathematics, University of Vienna, Oskar-Morgenstern-Platz 1, 1090 Vienna, Austria.} \\
\emph{E-mail address}: Wolfgang.Moens@univie.ac.at

\end{document}